\pdfoutput=1
\documentclass[12pt,reqno,a4wide]{amsart}

\usepackage[utf8]{inputenc}

\usepackage{tikz}
\usetikzlibrary{arrows,backgrounds,calc}
\usepackage{calrsfs}
\usepackage[charter]{mathdesign}
\usepackage{amsthm, amsmath, amsfonts}
\usepackage{graphicx}
\usepackage{hyperref}
\usepackage{lineno}
\usepackage{enumerate}
\usepackage{mathtools}
\usepackage{pgfplots}
\pgfplotsset{compat=1.5}

\usepackage[draft=false,kerning=true]{microtype}

\allowdisplaybreaks

\newtheorem{theorem}{Theorem}
\newtheorem{proposition}{Proposition}[section]
\newtheorem{corollary}[proposition]{Corollary}

\theoremstyle{remark}

\theoremstyle{definition}

\renewcommand{\MR}[1]{}

\newcommand{\TODO}[1]%
{\par\fbox{\begin{minipage}{0.9\linewidth}\textbf{TODO:} #1\end{minipage}}\par}

\DeclareSymbolFont{cmcal}{OMS}{cmsy}{m}{n}
\SetSymbolFont{cmcal}{bold}{OMS}{cmsy}{b}{n}
\DeclareSymbolFontAlphabet{\mathcal}{cmcal}

\newcommand{\C}[0]{\mathbb{C}}
\newcommand{\R}[0]{\mathbb{R}}

\renewcommand{\P}[0]{\mathbb{P}}
\newcommand{\E}[0]{\mathbb{E}}
\newcommand{\V}[0]{\mathbb{V}}

\DeclarePairedDelimiter{\abs}{\lvert}{\rvert}

\newcommand{\startingNecklace}{
\tikz{
    \node[draw, circle, inner sep=1.2pt] (A) {};
    \node[draw, circle, fill, inner sep=1.2pt, yshift=-0.6em] (B) {};
    \draw (A) to[bend right] (B);
    \draw (B) to[bend right] (A);
  }
}

\title[]{The Necklace Process: A Generating Function Approach}

\author[B.~Hackl]{Benjamin Hackl}

\address[Benjamin Hackl]
{Institut f\"ur Mathematik,
  Alpen-Adria-Uni\-ver\-si\-t\"at Klagenfurt, Universit\"atsstra\ss e
  65--67, 9020 Klagenfurt, Austria}
\email{\href{mailto:benjamin.hackl@aau.at}{benjamin.hackl@aau.at}}
\thanks{B.~Hackl is supported by the Austrian
  Science Fund (FWF): P~28466-N35 and by the Karl Popper Kolleg
  ``Modeling-Simulation-Optimization'' funded by the Alpen-Adria-Universit\"at Klagenfurt
  and by the Carinthian Economic Promotion Fund (KWF). This paper has been written while
  he was a visitor at Stellenbosch University.}

\author[H.~Prodinger]{Helmut Prodinger}
\address[Helmut Prodinger]{Department of Mathematical
  Sciences, Stellenbosch University, 7602 Stellenbosch,
 South Africa}
\email{\href{mailto:hproding@sun.ac.za}{hproding@sun.ac.za}}

\keywords{Necklace process; bivariate generating function; quasi-power theorem}
\subjclass[2010]{60C05; 05A15, 05A16}

\begin{document}

\begin{abstract}
  The ``necklace process'', a procedure constructing necklaces of black and white beads by
  randomly choosing positions to insert new beads (whose color is uniquely determined
  based on the chosen location), is revisited. This article illustrates how, after
  deriving the corresponding bivariate probability generating function, the
  characterization of the asymptotic limiting distribution of the number of beads of a
  given color follows as a straightforward consequence within the analytic
  combinatorics framework.
\end{abstract}
\maketitle

\section{Introduction}

We consider the following process (illustrated in Figure~\ref{fig:inserting-beads}) for
constructing necklaces with two-colored beads:
\begin{itemize}
\item[--] We start with \startingNecklace, the necklace with one black and one white bead.
\item[--] New beads can be added between any two adjacent beads. The color of the new bead
  is determined by the color of those two beads: The new bead is white if and only if its
  two neighbors are black.
\end{itemize}
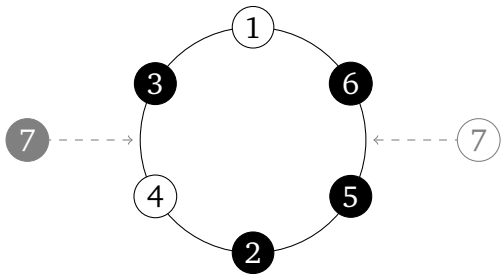
\begin{figure}[hb]
  \centering
  \begin{tikzpicture}
    \draw (0,0) circle (1.5cm);
    \draw (90: 1.5cm) node[draw, fill=white, circle, inner sep=2pt] {1};
    \draw (30: 1.5cm) node[draw, fill=black, text=white, circle, inner sep=2pt] {6};
    \draw (330: 1.5cm) node[draw, fill=black, text=white, circle, inner sep=2pt] {5};
    \draw (270: 1.5cm) node[draw, fill=black, text=white, circle, inner sep=2pt] {2};
    \draw (210: 1.5cm) node[draw, fill=white, text=black, circle, inner sep=2pt] {4};
    \draw (150: 1.5cm) node[draw, fill=black, text=white, circle, inner sep=2pt] {3};
    \draw[->, dashed, gray] (0:3cm) -- (0:1.6cm);
    \draw (0: 3cm) node[draw=gray, fill=white, text=gray, circle, inner sep=2pt] {7};
    \draw[->, dashed, gray] (180:3cm) -- (180:1.6cm);
    \draw (180: 3cm) node[draw=gray, fill=gray, text=white, circle, inner sep=2pt] {7};
  \end{tikzpicture}
  \caption{Construction of a necklace by the necklace process. The numbers on the beads
    correspond to the order in which they were inserted into the necklace.}
  \label{fig:inserting-beads}
\end{figure}

Motivated by a simple network communication model, this ``necklace process'' was analyzed
in~\cite{Mallows-Shepp:2008:necklace-process}. Further variants of this process were
discussed in~\cite{Nakata:2009:necklace-polya}, where an elegant approach using P\'olya urns to model the
edges connecting the beads was pursued. The parameters investigated in these
articles are the number of white beads in a random necklace of size $n$ (i.e., consisting
of $n$ beads), as well as the number of runs of black beads of given length.

The purpose of this brief note is to provide an alternative access to the analysis of the
number of beads of a given color in the necklace process by focusing on an appropriate generating
function and using tools from analytic combinatorics. A similar approach has been
successfully employed in, for example,
\cite{Morcrette:2012:analyzing-algebraic-polya,Morcrette-Mahmoud:2012:balanced-tenable-urns}.

In Section~\ref{sec:number-of-necklaces} we briefly discuss the combinatorial structure
surrounding the necklace process and elaborate how many different necklaces of given size
(i.e., consisting of a given amount of beads) can be constructed.
Then, in Section~\ref{sec:white-beads} we analyze the number of white and black beads in a
random necklace of given size. Our main result is given in Theorem~\ref{thm:explicit-bgf}, which
is an explicit formula for the bivariate probability generating function with respect to
the number of white beads---which has a surprisingly nice closed form.
Apart from some additional remarks on the structure of this
generating function, we then show in Corollaries~\ref{cor:white-beads:asy}
and~\ref{cor:black-beads:asy} how the qualitative results concerning the number of
black and white beads obtained in~\cite{Mallows-Shepp:2008:necklace-process} (expectation, variance,
limiting distribution) are a straightforward consequence of the explicitly known bivariate
generating function.

\section{Number of Necklaces}\label{sec:number-of-necklaces}
Before diving straight into the analysis of the number of beads of a given color, for the
sake of completeness, we briefly discuss the combinatorial structure of the objects we are
constructing.

While it is rather easy to see that there are $(n-1)!$ possible necklace
constructions\footnote{Starting with \startingNecklace, the necklace of size $2$, there
  are $2$ possible positions for a new bead. In the new necklace there are now $3$
  positions to choose from. Inductively, this proves that there are
  possible $(n-1)!$ construction processes for necklaces with $n$ beads.} for
a necklace of size $n$, many of those constructions yield the same necklace. Note that the
number of different necklaces of size $n$ is enumerated by sequence
\href{http://oeis.org/A000358}{A000358} in \cite{OEIS:2015}. We use the analytic
combinatorics framework in order to analyze this quantity in detail.

\begin{proposition}\label{prop:number-of-necklaces}
  Let $\mathcal{N}$ be the combinatorial class containing all different necklaces constructed by the
  necklace process. The corresponding ordinary generating function $N(z)$ enumerating these necklaces
  with respect to size is given by
  \begin{equation}\label{eq:number-gf}
    N(z) = \sum_{k\geq 1} \frac{\varphi(k)}{k} \log\Bigg(\frac{1 - z^{k}}{1 - z^{k}-
      z^{2k}}\Bigg),
  \end{equation}
  where $\varphi(k)$ is Euler's totient function.
  Asymptotically, the number of necklaces of size $n$ is given by
  \begin{equation}\label{eq:number}
    [z^{n}] N(z) = \Bigg(\frac{1 + \sqrt{5}}{2}\Bigg)^{n} n^{-1} + O\Bigg(\Bigg(\frac{1 +
      \sqrt{5}}{2}\Bigg)^{n/2} n^{-1}\Bigg).
  \end{equation}
\end{proposition}
\begin{proof}
  The generating function \eqref{eq:number-gf} can directly be obtained by means of the
  machinery provided by the symbolic method (see Chapter~I and in particular Theorem I.1
  in \cite{Flajolet-Sedgewick:ta:analy}). In fact, we can construct the combinatorial
  class $\mathcal{N}$ as
  \[ \mathcal{N} = \operatorname{Cyc}(\tikz[baseline=-0.3em]{\node[draw, circle, inner
      sep=3pt] {};} \times \tikz[baseline=-0.25em]{\node[draw, circle, fill, inner sep=3pt]
      {};}^{+}),  \]
  where $\tikz[baseline=-0.3em]{\node[draw, circle, inner sep=3pt] {};}$ and
  $\tikz[baseline=-0.3em]{\node[draw, circle, fill, inner sep=3pt] {};}^{+}$ represent the
  combinatorial classes for a single white and a non-empty sequence of black beads,
  respectively. Translating the construction of the combinatorial class $\mathcal{N}$ into
  the language of generating functions then immediately yields~\eqref{eq:number-gf}.

  In order to obtain the asymptotic growth of the coefficients of $N(z)$, we use
  singularity analysis (see~\cite{Flajolet-Odlyzko:1990:singul}, \cite[Chapter
  VI]{Flajolet-Sedgewick:ta:analy}), which requires us to identify the dominant
  singularities of $N(z)$, i.e., the singularities with minimal modulus.

  In fact, by observing that all $\zeta \in \C$ satisfying $\zeta^{k} = \frac{-1 \pm
    \sqrt{5}}{2}$  are roots of $1 - z^{k} - z^{2k} = 0$, it is easy to see that $N(z)$ has
  a unique dominant singularity located at $z = \frac{-1 + \sqrt{5}}{2}$ which comes from
  the first summand of $N(z)$ in~\eqref{eq:number-gf}. Extracting the coefficient growth
  provided by the first summand and observing that the singularity with the next-larger modulus
  is located at $z = \sqrt{(-1 + \sqrt{5}) / 2}$ (and comes from the second summand), we
  obtain~\eqref{eq:number}.
\end{proof}

\section{Beads of Equal Color}\label{sec:white-beads}

Let $n\geq 2$ and let $W_{n}$ and $B_{n}$ denote the random variables modeling the number
of white and black beads in a necklace of size $n$ that is constructed uniformly at
random, respectively.

The fact that $W_{n} + B_{n} = n$ allows us to concentrate our investigation on
$W_{n}$. Results from the characterization of $W_{n}$ can be translated directly to
$B_{n}$. Let $W(z,u)$ denote the shifted bivariate probability generating function
corresponding to $W_{n}$, that is
\[ W(z,u) = \sum_{n,k\geq 1} \P(W_{n} = k) z^{n-1} u^{k}.  \]
In contrast to previous works on the necklace process, we give an explicit formula for the
bivariate probability generating function $W(z,u)$.

\begin{theorem}\label{thm:explicit-bgf}
  The shifted bivariate probability generating function $W(z,u)$ corresponding to the
  random variables $W_{n}$ modeling the number of white beads in a random necklace of size
  $n$ is given by
  \begin{align}
    W(z,u) &= \frac{u}{\sqrt{1-u} \coth(z\sqrt{1-u}\,) - 1}, \label{eq:bgf}
  \intertext{or, equivalently, by}
    W(z,u) &= u \frac{\exp(z\alpha) - \exp(-z\alpha)}{\exp(z\alpha) (\alpha - 1) + \exp(-z\alpha)
             (\alpha + 1)} = u\frac{\exp(2z \alpha) - 1}{\exp(2z\alpha)(\alpha - 1) +
             \alpha + 1}, \label{eq:bgf-symmetric}
  \end{align}
  where $\alpha := \sqrt{1-u}$.
\end{theorem}
\begin{proof}
  Analogously to the approach in~\cite{Mallows-Shepp:2008:necklace-process} we also see
  the number of white beads in the context of a Markov chain with properties
  \begin{equation}\label{eq:markov}
    \P(W_{n+1} = k \mid W_{n} = k) = \frac{2 k}{n}, \qquad \P(W_{n+1} = k+1 \mid W_{n} =
    k) = 1 - \frac{2k}{n}.
  \end{equation}
  This is because when choosing the position for the new bead uniformly at random among
  all possible $n$ positions, the color of the bead to be inserted is white (which would
  increase the total number of white beads) if and only if
  we chose one of the positions not adjacent to a white bead. Given that there are $k$
  white beads among the $n$ beads, there are $n - 2k$ such positions.

  Let $w_{n,k} := \P(W_{n} = k)$ and let $w_{n}(u)$ denote the probability generating
  function for $W_{n}$. With the help
  of~\eqref{eq:markov} and the law of total probability we find
  \begin{align}\label{eq:prob-recurrence}
    \begin{split}
      w_{n+1,k} & = \P(W_{n+1} = k \mid W_{n} = k) w_{n,k} + \P(W_{n+1} = k \mid W_{n} = k-1)
      w_{n,k-1} \\
      & = \frac{2}{n} k w_{n,k} - \frac{2}{n} (k-1) w_{n,k-1} + w_{n,k-1}.
    \end{split}
  \end{align}
  It is interesting to note that structurally, these probabilities $w_{n,k}$ are strongly
  connected to Eulerian numbers (see~\cite[Section~5.1.3.]{Knuth:1998:Art:3}): By setting
  $e_{n,k} := w_{n,k}/(n-1)!$ in~\eqref{eq:prob-recurrence} we obtain
  \[ e_{n,k} = 2k e_{n-1, k} + (n+1 - 2k) e_{n-1, k-1},  \]
  which strongly resembles the recurrence for Eulerian numbers as given in
  \cite[5.1.13.(2)]{Knuth:1998:Art:3}.

  By multiplication of~\eqref{eq:prob-recurrence} with $u^{k}$ and summing over $k$ this translates to
  \begin{equation}\label{eq:recurrence}
    w_{n+1}(u) = \frac{2 u (1-u)}{n} w_{n}'(u) + u w_{n}(u)
  \end{equation}
  for $n \geq 2$ with initial value $w_{2}(u) = u$. Note that the bivariate probability
  generating function $W(z,u)$ can be expressed by the $w_{n}(u)$ by means of $W(z,u) =
  \sum_{n\geq 2} w_{n}(u) z^{n-1}$. After multiplying~\eqref{eq:recurrence} with $z^{n-1}$ and
  summation over $n\geq 2$ we find
  that $W(z,u)$ satisfies the first order linear partial differential equation
  \begin{equation}\label{eq:pde}
    \partial_{z} W(z,u) (1 - zu) = 2u(1-u) \partial_{u} W(z,u) + u W(z,u) +
    u,
  \end{equation}
  with the condition $W(0,u) = 0$. Solving this PDE (e.g., by means of the method of
  characteristics, or with the help of a computer algebra system)
  yields~\eqref{eq:bgf}. The alternate form~\eqref{eq:bgf-symmetric} follows from
  rewriting $\coth(z) = \frac{\exp(z) + \exp(-z)}{\exp(z) - \exp(-z)}$.
\end{proof}

Because of the particularly nice shape of the bivariate probability generating function
$W(z,u)$ we are able to use the machinery around Hwang's quasi-power theorem
(see~\cite{Hwang:1998}, \cite[Section IX.7]{Flajolet-Sedgewick:ta:analy}) in order to
find the characterization of $W_{n}$ from~\cite{Mallows-Shepp:2008:necklace-process} as an
immediate corollary.

\begin{corollary}[{\cite[Section
    3]{Mallows-Shepp:2008:necklace-process}}]\label{cor:white-beads:asy}
  The expected number of white beads in a necklace of size $n$ constructed uniformly at
  random and the corresponding variance are given by
  \begin{equation}\label{eq:number-exp-var}
    \E W_{n} = \frac{n}{3}\qquad \text{ and }\qquad \V W_{n} = \frac{2n}{45}
  \end{equation}
  for $n \geq 6$.
  Furthermore, $W_{n}$ is asymptotically normally distributed in the sense that for all
  $x\in \R$ we have
  \begin{equation}\label{eq:beads-normal}
    \P\Bigg(\frac{W_{n} - n/3}{\sqrt{2n/45}} \leq x\Bigg) = \frac{1}{\sqrt{2\pi}}
    \int_{-\infty}^{x} e^{-t^{2}/2}~dt + O(n^{-1/2}).
  \end{equation}
\end{corollary}
\begin{proof}
  This explicit form of the bivariate probability generating function allows us to use
  standard techniques from analytic combinatorics in order to obtain the expected value
  $\E W_{n}$ and the variance $\V W_{n}$. By construction, $\E W_{n}$ is the coefficient
  of $z^{n-1}$ in $\partial_{u} W(z,1)$. We find
  \[ \partial_{u} W(z,1) = \frac{(z^{2} - 3z - 3) z}{3 (1 - z)^{2}} = z +
    \sum_{n\geq 3} \frac{n}{3} z^{n-1},  \]
  which proves $\E W_{n} = n/3$ for $n\geq 3$. Similarly, we can use the second partial
  derivative of $W(z,u)$ with respect to $u$ in order to extract the second factorial
  moment, $\E(W_{n}(W_{n}-1))$. Together with the well-known identity
  \[ \V W_{n} = \E(W_{n}(W_{n}-1)) + \E W_{n} - (\E W_{n})^{2}  \]
  this allows to verify that $\V W_{n} = 2n/45$ for $n\geq 6$.

  Finally, the asymptotically normal limiting distribution is obtained immediately by
  using the explicit formula for $F(z,u) := zW(z,u)$ and applying~\cite[Theorem
  IX.12]{Flajolet-Sedgewick:ta:analy}. The corresponding necessary conditions are all
  checked easily:
  \begin{itemize}
  \item[--] \emph{Analytic perturbation.} We have
    \[ A(z,u) = 0 ,\quad B(z,u) = z^{2}u,\quad C(z,u)
      = z(\sqrt{1 - u}\coth(z\sqrt{1-u}) - 1) \] and $\alpha = 1$. $A(z,u)$ and $B(z,u)$ are
    obviously entire functions, and $C(z,u)$ is analytic in the domain $\{z\in \C :
    \abs{z} < 2\pi\} \times \{u\in \C : \abs{u-1} < \varepsilon\}$ for some $\varepsilon >
    0$. To see this, observe that $\coth(z)$ is an odd function with Laurent expansion
    \[ \coth(z) = \frac{1}{z} + \frac{z}{3} - \frac{z^{3}}{45} + \frac{2z^{5}}{945} +
      O(z^{7}),  \]
    which means that only even powers of $\sqrt{1-u}$ occur. Hence, there is no branching
    point of the square root at $u = 1$. Also, $C(z,1) = 1-z$ has a unique simple root at
    $\rho = 1$ with $B(\rho,1) = 1 \neq 0$.
  \item[--] \emph{Non-degeneracy.} It is straightforward to check $\partial_{z} C(\rho,1)
    \cdot \partial_{u}C(\rho, 1) = 1/3 \neq 0$.
  \item[--] \emph{Variability.} By explicitly computing the variance above we already
    computed that the linear term does not vanish.
  \end{itemize}
  This proves~\eqref{eq:beads-normal} and concludes this proof.
\end{proof}

As mentioned above, the characterization $W_{n}$ can immediately be carried over to
$B_{n}$. Rewriting $B_{n} = n - W_{n}$ proves the following result.
\begin{corollary}\label{cor:black-beads:asy}
  The expected number of black beads in a necklace of size $n$ constructed uniformly at
  random and the corresponding variance are given by
  \begin{equation}\label{eq:black-bleads-exp-var}
    \E B_{n} = \frac{2n}{3},\qquad \text{ and }\qquad \V B_{n} = \frac{2n}{45}
  \end{equation}
  for $n\geq 6$. Furthermore, $B_{n}$ is asymptotically normally distributed.
\end{corollary}

As a side effect of Theorem~\ref{thm:explicit-bgf} we are also able to extract more
information on the probability generating functions $w_{n}(u)$.
\begin{corollary}
  Let $\alpha = \sqrt{1 - u}$. The shifted probability generating functions $w_{n}(u)/u$
  can be expressed by means of even polynomials $r_{n}(\alpha)$ satisfying the recurrence
  relation
  \begin{equation}\label{eq:prob-recurrence-2}
    r_{n}(\alpha) = \frac{(2\alpha)^{n-2}}{(n-1)!} + (1-\alpha) \sum_{k=0}^{n-2}
    \frac{(2\alpha)^{k}}{(k+1)!} r_{n-1-k}(\alpha),
  \end{equation}
  for $n \geq 2$ with initial value $r_{1}(\alpha) = 0$.
\end{corollary}
\begin{proof}
  From~\eqref{eq:bgf-symmetric}, and by the definition of $W(z,u)$, we find
  \[ \frac{\exp(2z\alpha) - 1}{\exp(2z\alpha) (\alpha - 1) + \alpha + 1} = \sum_{n\geq 2}
    z^{n-1} w_{n}(u)/u = \sum_{n\geq 2} z^{n-1} r_{n}(\alpha).  \]
  After multiplying with the denominator and extracting the coefficient of $z^{n-1}$ on
  both sides, we obtain
  \[ \frac{(2\alpha)^{n-1}}{(n-1)!} = (\alpha + 1) r_{n}(\alpha) + (\alpha - 1) \sum_{k=0}^{n-1}
    \frac{(2\alpha)^{k}}{k!} r_{n-k}(\alpha),  \]
  valid for $n\geq 2$. Rearranging this equation then yields~\eqref{eq:prob-recurrence-2}.
\end{proof}

\bibliographystyle{bibstyle/amsplainurl}
\bibliography{cheub-bib/cheub}

\end{document}